\documentclass[11pt,a4paper,oneside]{amsart}

\usepackage{amssymb,amsmath}
\usepackage[dvips]{graphics}
\usepackage[dvips]{color}
\usepackage{graphicx}
\usepackage{hyperref}
\usepackage[english]{babel}
\usepackage[T1]{fontenc}       		%uncomment these two lines
\usepackage{mathrsfs}   
\usepackage{comment}
\usepackage[textsize=tiny]{todonotes}
\usepackage{xcolor}

\newtheorem{Theorem}{Theorem}
\newtheorem{Lemma}[Theorem]{Lemma}
\newtheorem{Proposition}[Theorem]{Proposition}
\newtheorem{Corollary}[Theorem]{Corollary}
\newtheorem{Conjecture}[Theorem]{Conjecture}

\theoremstyle{definition}
\newtheorem{Definition}[Theorem]{Definition}
\newtheorem{Example}[Theorem]{Example}
\newtheorem{Remark}[Theorem]{Remark}
\newtheorem{Claim}{Claim}
\providecommand{\ch}[1]{\text{\raise 2pt \hbox{$\chi$}\kern-0.2pt}_{#1}}

\def\XXint#1#2#3{{\setbox0=\hbox{$#1{#2#3}{\int}$}
      \vcenter{\hbox{$#2#3$}}\kern-.5\wd0}}

\newcommand{\N}{\mathbb{N}}     %natural numbers
\newcommand{\R}{\mathbb{R}}     %real numbers
\newcommand{\Z}{\mathbb{Z}}         %integers

\newcommand{\calB}{\mathscr{B}}

\newcommand{\calQ}{\mathscr{Q}}

\newcommand{\bsigma}{\boldsymbol{\sigma}}

\newcommand{\btheta}{\boldsymbol{\theta}}

\DeclareMathOperator{\diam}{diam}			%diameter
\usepackage[mathcal]{eucal}					%for the set of BV subsets of Rn

\renewcommand{\geq}{\geqslant}
\renewcommand{\leq}{\leqslant}
\renewcommand{\epsilon}{\varepsilon}

\usepackage{pgf,tikz}
\usepackage{mathrsfs}
\usetikzlibrary{arrows}
\definecolor{qqwwtt}{rgb}{0.,0.4,0.2}
\definecolor{ffqqqq}{rgb}{1.,0.,0.}
\definecolor{qqzzff}{rgb}{0.,0.6,1.}
\definecolor{qqzzqq}{rgb}{0.,0.6,0.}
\definecolor{ccqqqq}{rgb}{0.8,0.,0.}
\definecolor{ttttff}{rgb}{0.2,0.2,1.}

\begin{document}
\author{Emma D'Aniello and Laurent Moonens}
\date{\today}
\thanks{Laurent Moonens acknowledges the partial support of the ``Luigi Vanvitelli'' University of Campania, through a scholarship awarded to him as a ``Visiting Professor''.}

\title[Differentiation along rectangles]{Differentiating along rectangles\\
with fixed shapes in a set of directions}

\maketitle

\begin{abstract}
In the present note, we examine the behavior of some homo\-thecy-invariant differentiation basis of rectangles in the plane satisfying the following requirement: for a given rectangle to belong to the basis, the ratio of the largest of its side-lengths by the smallest one (which one calls its \emph{shape}) has to be a fixed real number depending on the angle between its longest side and the horizontal line (yielding a \emph{shape-function}). Depending on the allowed angles and the corresponding shape-function, a basis may differentiate various Orlicz spaces. We here give some examples of shape-functions so that the corresponding basis differentiates $L\log L(\R^2)$, and show that in some ``model'' situations, a fast-growing shape function (whose speed of growth depends on $\alpha>0$) does not allow the differentiation of $L\log^\alpha L(\R^2)$.
\end{abstract}

A (Buseman-Feller) \emph{differentiation basis} in the plane is a collection $\calB$ of open sets such that for every $x\in\R^2$ one has $\inf\{\diam B : B\in\calB, B\ni x\}=0$. In the sequel, we shall always assume that $\calB$ is \emph{homothecy-invariant}, meaning that one has $x+\lambda B\in \calB$ for all $x\in \R^2$ and all $\lambda>0$.

It can arise, given a locally integrable function $f\in L^1_{\mathrm{loc}}(\R^2)$, that Lebesgue's differentiation theorem's conclusion holds for $f$ using sets in $\calB$ instead of the usual Euclidean balls, namely that one has, for a.e. $x\in\R^2$:
$$
f(x)=\lim_{\begin{subarray}{c} x\in R\in \calB\\\diam R\to 0\end{subarray}} \frac{1}{|B|} \int_B f.
$$
In case the latter holds, we shall say that $\calB$ \emph{differentiates} $\int f$. If $\calB$ differentiates $\int f$ for all $f\in X\subseteq L^1_{\mathrm{loc}}(\R^2)$, we shall than say, for simplicity, that $\calB$ \emph{differentiates $X$}. If a basis differentiates $L^\infty(\R^2)$, we call it a \emph{density basis}.

It is well-known (see \emph{e.g.} De Guzmán's book \cite{DEGUZMAN1975}) that, for $\calB$ a homothecy-invariant basis as above, the two following properties are equivalent for a given Young function $\Phi:\R_+\to\R_+$ (one ofter refers to this as the \emph{Sawyer-Stein} principle for differentiation bases)
\begin{itemize}
\item[(i)] $\calB$ differentiates the Orlicz space $L^\Phi(\R^2)$;
\item[(ii)] the maximal operator $M_{\calB}$ defined by $M_{\calB} f(x):=\sup_{x\in B\in\calB} \frac{1}{|B|} \int_B |f|$ satisfies the following estimate for all measurable $f$ and all $\lambda>0$:
\begin{equation}\label{eq.maxphi}
|\{x\in\R^2: M_{\calB} f(x)>\lambda\}|\leq C\int_{\R^2} \Phi\left(\frac{|f|}{\lambda}\right).
\end{equation}
\end{itemize}

When $\calB$ is a homothecy-invariant collection of \emph{rectangles}, various situations can occur; here are some famous examples:
\begin{itemize}
\item if $\calB$ is the collection of all rectangles parallel to the axes, then $L\log L(\R^2)$ is the largest Orlicz space that $\calB$ differentiates (see Stokolos \cite{STOKOLOS1988});
\item if $\calB$ is the collection of all rectangles (parallel to the axes or not), there always exists $f\in L^\infty(\R^2)$ such that $\int f$ is \emph{not} differentiated by $\calB$ (see Buseman and Feller \cite{BUSEMANNFELLER})~---~actually the same conclusion holds even if one replaces the collection of all rectangles by the collection of rectangles one side of which makes an angle with the horizontal line belonging to \emph{some} countable sets, like \emph{e.g.} any set $\btheta \subseteq [0,2\pi)$ that is dense in some interval (as it follows from \cite{BUSEMANNFELLER}) or even $\btheta=\{\frac 1n:n\in\N^*\}$ (see De Guzmán \cite{DEGUZMAN1981});
\item if, though, $\calB$ is the set of all rectangles one side of which makes an angle with the horizontal line belonging to the set $\btheta=\{2^{-k}:k\in\N\}$ (or to the image of any lacunary sequence as defined in \cite{MOONENS2016}), then $\calB$ is known to differentiate $L^p(\R^2)$ for all $1<p\leq \infty$ (see Córdoba and Fefferman \cite{CF1977} for $p\geq 2$ and Nagel, Stein and Wainger \cite{NSW1978} for all $p>1$);
\item it actually follows from a beautiful paper by Bateman \cite{BATEMAN} that if $\btheta\subseteq [0,2\pi)$ is a given set of angles, then the basis $\calB_{\btheta}$ of all rectangles one side of which makes an angle $\theta\in\btheta$ with the horizontal line, either differentiates all $L^p(\R^2)$ for all $1<p<\infty$, or fails to differentiate any $L^p(\R^n)$ for $1<p<\infty$ (dichotomy which, as we observed with J.M.~Rosenblatt in \cite{DMR}, remains true if one replaces the range $1<p<\infty$ by $1<p\leq \infty$);
\item it also follows from \cite{BATEMAN} that $\calB_{\btheta}$ never differentiates any $L^p(\R^n)$ for $1<p<\infty$ (and hence never is a density basis according to \cite{DMR}) if $\btheta$ is uncountable.
\end{itemize}

Defining the \emph{shape} $\sigma(R)$ of a rectangle $R$ as the quotient of its longest side-length by its shortest side-length, positive or negative differentiation results can also depend on restrictions made on this ratio. Let us mention, for example, a few situations where this influence is well understood:
\begin{itemize}
\item if $\calB$ is a homothecy-invariant basis of rectangles whose shapes are bounded from above, then $\calB$ differentiates $L^1(\R^2)$ (this easily follows from the fact that the maximal operator $M_{\calB}$ then behaves distributionally like the uncentered Hardy-Littlewood maximal operator on balls);
\item if for any $n\in\N^*$, $C_n$ denotes the set of all ternary numbers of the form $\sum_{j=1}^n a_j3^{-j}$ for some $a_j\in\{0,2\}$, $1\leq j\leq n$ (which one can see as a ``truncated ternary Cantor set''), then the basis $\calB$ of all rectangles $R$ such that, for some $n\in\N^*$, $R$ has shape $\frac 1n$ and has its longest side making an angle $\theta$ satisfying $\tan\theta\in C_n$, fails to differentiate $L^p(\R^2)$ for all $1\leq p\leq 2$, as it follows from Katz \cite{KATZ1996}.
\end{itemize}
In this short note, we focus our attention on the following question: given an infinite set of angles $\btheta\subseteq [0,2\pi)$ such that $0$ is a limit point of $\btheta$, does there exist a function $\bsigma:\btheta\to [1,+\infty)$ satisfying $\sup\bsigma=+\infty$, for which the basis $\calB^{\bsigma}$ differentiates (or fails to differentiates) a given Orlicz space, where $\calB^{\bsigma}$ denotes the basis of all rectangles for which there exists $\theta\in\btheta$ so that the longest side of $R$ makes an angle $\theta$ with the horizontal line, and that $\sigma(R)=\bsigma(\theta)$? We here examinate the following situations:
\begin{itemize}
\item under the above assumptions, the basis $\calB^{\bsigma}$ \emph{never} differentiates $L^1(\R^2)$ as it follows from a result by Moriyón (see Proposition~\ref{prop.1} below);
\item given $\epsilon_0>0$ small enough and any set $\btheta\subseteq [0,\epsilon_0]$, there always exists a (nonincreasing) function $\bsigma$ satisfying the above conditions and for which $\calB^{\bsigma}$ differentiates \emph{exactly} $L\log L(\R^2)$ (see Corollary~\ref{cor}, which basically follows from a simple geometric observation we describe in section~\ref{sec.2}, and from a result by Stokolos \cite{STOKOLOS1988});
\item if $\btheta$ is obtained from the ``model'' geometrical sequence $(2^{-k})_{k\in\N}$ by inserting uniformly $N_k$ angles in between $2^{-k-1}$ and $2^{-k}$, and if $\bsigma$ is constant on each of those ``blocks'', then, depending on how $(N_k)$, $(\sigma_k)$ and the Young function $\Phi$ behave with respect to each other, it may happen that the associated basis $\calB^{\bsigma}$ fails to differentiate the Orlicz space $L^\Phi(\R^2)$ (recall that it is still an open problem whether $\calB_{\theta_0}$ does or does not differentiate $L\log^l L(\R^2)$ for $l\geq 1$); here of course $(2^{-k})_{k\in\N}$ could be replaced by any lacunary sequence in the sense of \cite{MOONENS2016}.
\end{itemize}

\section{Using a result by R.~Moriyón}

We keep the notations defined in the introduction.
\begin{Proposition}\label{prop.1}
Assume that $\btheta\subseteq [0,2\pi)$ and $\bsigma: \btheta\to [1,+\infty)$ satisfy $\sup\bsigma=+\infty$. In this case, the homothecy-invariant basis $\calB^{\bsigma}$ fails to differentiate $L^1(\R^2)$.
\end{Proposition}
\begin{proof}
Define, as in Moriyón's theorem (cited in \cite[Appendix III, p.~206]{DEGUZMAN1975}), the set:
$$
K:=\cup\{R\in \calB^{\bsigma}: R\ni 0, |R|\leq 1\}.
$$
Using the fact that $\sup\bsigma=+\infty$, choose $(\theta_k)\subseteq [0,2\pi)$ for which $\bsigma(\theta_k)\to+\infty$. Define then $R_k$ as the rectangle with area $1$ centered at the origin, having shape $\bsigma(\theta_k)$ and its longest side making an angle $\theta_k$ with the horizontal line, so that one has $R_k\in\calB^{\bsigma}$, $0\in R_k$, $|R_k|=1$ and hence also $R_k\subseteq K$. The length of its longest side being equal to $\sqrt{\bsigma(\theta_k)}$, which can be arbitrary large, it is clear that $K$ is unbounded. It hence follows from condition e) in Moriyón's theorem (cited in \cite[Appendix III, p.~206]{DEGUZMAN1975}) that $\calB^{\bsigma}$ fails to differentiate $L^1(\R^2)$.
\end{proof}

\section{From a simple geometrical observation to bases differentiating $L\log L(\R^2)$}\label{sec.2}

\subsection{A simple geometrical observation} Fix a rectangle $R$ with longest side $L$ and shortest side $\ell$, such that its longest side makes an angle $\theta$ with the origin. Fix then a parameter $t\in ]0,\frac 12 [$ and denote by $\check{R}$ the rectangle parallel to the axes contained inside $R$ and determined by the fact that two opposite vertices meet the pair of longest sides of $R$ at points distant of $tL$ from the nearest vertex of $R$ lying on the same side (see Figure~\ref{fig.1} below). Denote by $\check{L}$ and $\check{\ell}$ the horizontal and vertical sides of $\check{R}$. On the other hand, denote by $\hat{R}$ the smallest rectangle parallel to the axes containing $R$, and call $\hat{L}$ and $\hat{\ell}$ its horizontal and vertical sides, respectively.

\begin{figure}[h!]
\begin{center}
\begin{tikzpicture}[line cap=round,line join=round,>=triangle 45,x=1.1cm,y=1.35cm]
\clip(-1.28066891160634,-2.10094409662089) rectangle (10.544084683564193,3.6775675659247553);
\fill[line width=2.8pt,color=qqzzff,fill=qqzzff,fill opacity=0.10000000149011612] (0.,0.) -- (8.955037487502233,0.8985007498214519) -- (8.855204070855406,1.8935049150994776) -- (-0.09983341664682799,0.9950041652780258) -- cycle;
\fill[line width=2.pt,color=ffqqqq,fill=ffqqqq,fill opacity=0.10000000149011612] (2.1389259552287303,1.2196293527333888) -- (6.716278115626675,1.2196293527333888) -- (6.716278115626675,0.6738755623660889) -- (2.1389259552287303,0.6738755623660889) -- cycle;
%\fill[line width=2.pt,color=qqwwtt] (-0.09983341664682799,0.) -- (8.955037487502233,0.) -- (8.955037487502233,1.8935049150994776) -- (-0.09983341664682799,1.8935049150994776) -- cycle;
\draw [line width=2.pt,color=ffqqqq] (2.1389259552287303,1.2196293527333888)-- (6.716278115626675,1.2196293527333888);
\draw [line width=2.pt,color=ffqqqq] (6.716278115626675,1.2196293527333888)-- (6.716278115626675,0.6738755623660889);
\draw [line width=2.pt,color=ffqqqq] (6.716278115626675,0.6738755623660889)-- (2.1389259552287303,0.6738755623660889);
\draw [line width=2.pt,color=ffqqqq] (2.1389259552287303,0.6738755623660889)-- (2.1389259552287303,1.2196293527333888);
\draw [line width=2.pt,color=qqwwtt] (-0.09983341664682799,0.)-- (8.955037487502233,0.);
\draw [line width=2.pt,color=qqwwtt] (8.955037487502233,0.)-- (8.955037487502233,1.8935049150994776);
\draw [line width=2.pt,color=qqwwtt] (8.955037487502233,1.8935049150994776)-- (-0.09983341664682799,1.8935049150994776);
\draw [line width=2.pt,color=qqwwtt] (-0.09983341664682799,1.8935049150994776)-- (-0.09983341664682799,0.);
\draw [line width=2.8pt,dash pattern=on 1pt off 1pt] (-0.09983341664682799,0.9950041652780258)-- (2.1389259552287303,1.2196293527333888);
\draw [line width=2.pt,color=qqzzff] (-0.09983341664682799,0.9950041652780258)-- (0.,0.);
\draw [line width=2.pt,color=qqzzff] (0.,0.)-- (8.955037487502233,0.8985007498214519);
\draw [line width=2.pt,color=qqzzff] (-0.09983341664682799,0.9950041652780258)-- (8.855204070855406,1.8935049150994776);
\draw [line width=2.pt,color=qqzzff] (8.855204070855406,1.8935049150994776)-- (8.955037487502233,0.8985007498214519);
\begin{scriptsize}
\draw[color=qqzzff] (4.163179913340095,1.5413031192500621) node {$R$};
\draw[color=ffqqqq] (5.635696398776425,0.9277545836515863) node {$\check{R}$};
\draw[color=ffqqqq] (4.609397030138982,0.7938894486119188) node {$\check{L}$};
\draw[color=ffqqqq] (2.5233320091041804,0.916599155731614) node {$\check{\ell}$};
\draw[color=qqwwtt] (5.323344417017203,-0.29934248754536547) node {$\hat{L}$};
\draw[color=qqwwtt] (9.30583218444728,0.40344947141288867) node {$\hat{\ell}$};
\draw[color=qqwwtt] (3.7950507919810117,2.154851654848538) node {$\hat{R}$};
\draw[color=black] (0.9838829561480169,1.3181945608506165) node {$tL$};
\draw[color=qqzzff] (0.12491500631015737,0.5484700343725285) node {$\ell$};
\draw[color=qqzzff] (4.508998178859232,0.23611805261330437) node {$L$};
\end{scriptsize}
\end{tikzpicture}
\caption{The rectangles $R$, $\check{R}$ and $\hat{R}$}\label{fig.1}
\end{center}
\end{figure}
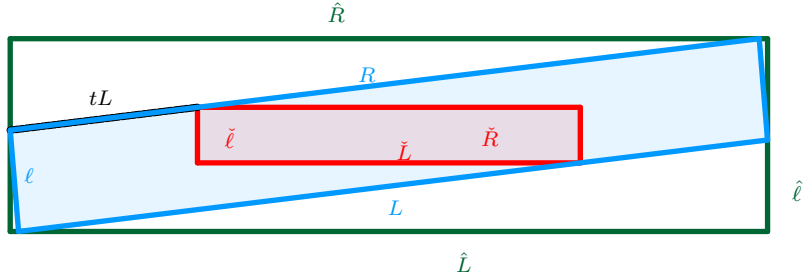

Simple trigonometric computations yield:
$$
\check{L}=(1-2t)L\cos\theta+\ell\sin\theta,
$$
$$
\check{\ell}= (1-2t)L\sin\theta+\ell\cos\theta,
$$
$$
\hat{L}=L\cos\theta+\ell\sin\theta,
$$
and:
$$
\hat{\ell}=\ell\cos\theta+L\sin\theta.
$$

Letting $\sigma:=L/\ell$ denote the {shape} of $R$, it's now a routine computation to calculate:
$$
\hat{A}_t(\theta,\sigma):=|\hat{R}|=\left[1+\frac 12 \left(\sigma+\frac{1}{\sigma}\right)\sin 2\theta\right]|R|,
$$
and:
$$
\check{A}_t(\theta,\sigma):=|\check{R}|=\left\{(1-2t)\cos 2\theta+\frac 12 \left[ \frac{1}{\sigma}-(1-2t)^2\sigma\right]\sin 2\theta\right\} |R|.
$$
Define now the ratio:
$$
\rho_t(\theta,\sigma):=\frac{|\hat{R}|}{|\check{R}|} =\frac{\hat{A}_t(\theta,\sigma)}{\check{A}_t(\theta,\sigma)}.
$$

Assuming that $\theta$ is small enough, observe that one has $\check{A}_t(\theta,\sigma)>0$ (so that the construction makes sense) provided that one has:
$$
1\leq\sigma <\sigma^*_{t,\theta}:=\frac{1}{1-2t} \left(\cot 2\theta+\sqrt{1+\cot^2 2\theta}\right).
$$
On the other hand, one computes:
\begin{multline*}
[\check{A}_t(\theta,\sigma)]^2\partial_\sigma\rho_t(\theta,\sigma)
=\frac{1}{2\sigma}\sin^2 2\theta [1+(1-2t)^2]\\+\frac{1-2t}{4}\sin(4\theta)\left(1-\frac{1}{\sigma^2}\right)+\frac 12 \sin 2\theta \left[\frac{1}{\sigma^2}+(1-2t)^2\right]>0,
\end{multline*}
so that in the above range for $\sigma$, the function $\sigma \mapsto \rho_t(\theta,\sigma)$ is increasing if $\theta$ is fixed (and small enough) and tends to $+\infty$ as $\sigma$ approaches $\sigma^*_{t,\theta}$.

Finally, one gets also, for $\frac{1}{1-2t}\leq\sigma <\sigma^*_{t,\theta}$:
\begin{multline*}
[\check{A}_t(\theta,\sigma)]^2\partial_\theta\rho_t(\theta,\sigma)\\
=(1-2t)\left(\sigma+\frac{1}{\sigma}\right)+2(1-2t)\sin 2\theta +\left[(1-2t)^2\sigma-\frac{1}{\sigma}\right] \cos 2\theta>0,
\end{multline*}
so that $\theta \mapsto \rho_t(\theta,\sigma)$ is also an increasing map in the latter range for $\sigma$.

Now fix a set $\btheta\subseteq [0,\pi/6)$ of which $0$ is a limit point, fix a large number $\rho_0\geq 4\left(\frac{1-t}{1-2t}\right)^2$ and choose, for any $\theta\in\btheta$, a ratio $\frac{1}{1-2t}\leq \bsigma(\theta)< \sigma^*_{t,\theta}$ such that one has:
$$
\rho_t[\theta,\bsigma(\theta)]=\rho_0\ ;
$$
note that this is possible since we have, for all $0<\theta\leq \pi/6$:
$$
\rho\left(\theta, \frac{1}{1-2t}\right)\leq 4\left(\frac{1-t}{1-2t}\right)^2.
$$
Observe, using what has been said before, that $\bsigma:\btheta\to (1,\infty), \theta\mapsto \bsigma(\theta)$ is decreasing. Computing moreover the value of $\bsigma(\theta)$ in terms of $\theta$ (and $t$), we get easily:
$$
\bsigma(\theta)\geq \frac{(1-2t)\rho_0-\frac{1}{\cos 2\theta}}{(1-2t)^2\rho_0+1} \cot 2\theta,
$$
which yields $\bsigma(\theta)\to\infty$, $\theta\to 0$.

\subsection{Constructing a differentiation basis} Given $\btheta\subseteq [0,\pi/6)$ for which $0$ is a limit point, we constructed in the previous section a function $\bsigma: \btheta\to (1,+\infty)$. Let now $\calB^{\bsigma}$ be the basis, defined in the introduction, of all rectangles $R$ for which there exists a $\theta\in\btheta$ so that $\sigma(R)=\bsigma(\theta)$ and that the longest side of $R$ makes an angle $\theta$ with the horizontal line. Define also two basis of two-dimensional intervals by:
$$
\hat{\calB}^{\bsigma}:=\{\hat{R}:R\in\calB^{\bsigma}\}\quad\text{and}\quad\check{\calB}^{\bsigma}:=\{\check{R}:R\in\calB^{\bsigma}\}.
$$

\begin{Lemma}\label{lem.1}
Given a measurable function $f$, one has:
$$
\frac{1}{\rho_0}M_{\check{\calB}^{\bsigma}}f \leq M_{\calB^{\bsigma}}f\leq \rho_0 M_{\hat{\calB}^{\bsigma}}f.
$$
\end{Lemma}
\begin{proof}
Assume first that $R\in \calB^{\bsigma}$ is given and compute, using the previous notations:
$$
\frac{1}{|R|} \int_R |f| \leq \frac{|\hat{R}|}{|R|} \frac{1}{|\hat{R}|}\int_{\hat{R}} |f|\leq \rho_0  \frac{1}{|\hat{R}|}\int_{\hat{R}} |f|.
$$
We hence have:
$$
M_{\calB^{\bsigma}}f\leq \rho_0 M_{\hat{\calB}^{\bsigma}}.
$$
If now one fixes $Q\in\check{\calB}^{\bsigma}$, then denote by $R\in \calB^{\bsigma}$ a rectangle satisfying $\check{R}=Q$. One then writes:
$$
\frac{1}{|Q|}\int_Q |f|=\frac{1}{|\check{R}|}\int_{\check{R}}|f|\leq \rho_0 \frac{1}{|R|} \int_R |f|,
$$
and we hence get $M_{\check{\calB}^{\bsigma}}f\leq \rho_0 M_{\calB^{\bsigma}}$.
\end{proof}
\begin{Corollary}\label{cor}
The basis $\calB^{\bsigma}$ differentiates \emph{exactly} $L\log L(\R^2)$.
\end{Corollary}

To prove this corollary, we shall need the following lemma, relying mainly on \cite{STOKOLOS1988}.
\begin{Lemma}\label{lem.2}
Assume that $\calB$ is a homothecy-invariant, Buseman-Feller differentiation basis of rectangles parallel to the coordinate axes in $\R^2$. If $\sup\{\sigma(R):R\in \calB\}=+\infty$, then $\calB$ differentiates exactly $L\log L(\R^2)$.
\end{Lemma}
\begin{proof}[Proof of the lemma]
To prove this lemma, observe that one can, without loss of generality, assume that all rectangles in $\calB$ have their longest side parallel to the $x$-axis. It is clear indeed, that one can write $\calB=\calB_x\cup\calB_y$, where $\calB_x$ and $\calB_y$ are the collection of elements of $\calB$ whose longest side lie in the $x$- and $y$- direction, respectively. One has then $\sup\{\sigma(R):R\in\calB_i\}=+\infty$ for at least one $i\in\{x,y\}$. It then follows from symmetrization if necessary, than one cas always assume $i=x$. Now just note that if $\calB_x$ differentiates exactly $L\log L(\R^2)$, then the same is true for $\calB$.

As in \cite{STOKOLOS1988}, denote by $\calB^*$ the set of all dyadic parents of elements in $\calB$ (the \emph{dyadic parent} $R^*$ of a rectangle $R$ parallel to the axes, being the rectangle with dyadic side-lenths containing $R$, concentric with it and having the smallest possible area). Since it is easy to check that one has $\sigma(R^*)\geq\frac 12 \sigma(R)$ for any rectangle $R$ parallel to the axes, is now clear that one has $\sup\{\sigma(R^*):R\in\calB\}=+\infty$. It is also easy to see that $\calB^*$ is translation invariant. Finally, observe that $\calB^*$ is also invariant under homothecies with dyadic ratio. Indeed, fix $R\in\calB(0)$ and $k\in\Z$ and let's see that $2^k R^*\in\calB^*$. We need to establish that $2^k R^*=Q^*$ for some $Q\in\calB$. Yet if one denotes by $x_0$ and $y_0$ (resp. $x_0^*$ and $y_0^*$) the lenths of the $x$- and $y$-sides of $R$ (resp. $R^*$) respectively, we get by definition of the dyadic parent $\frac 12 x_0^*<x_0\leq x_0^*$ and $\frac 12 y_0^*<y_0\leq y_0^*$. This also yields  $\frac 12 2^kx_0^*<2^kx_0\leq 2^kx_0^*$ and $\frac 12 2^ky_0^*<2^ky_0\leq 2^ky_0^*$, meaning that $(2^k R)^*$ has side-lengths $2^kx_0^*$ and $2^ky_0^*$. Denoting by $Q$ the rectangle parallel to the axes with same center as $(2^kR)^*$ and side-lengths $2^kx_0$ and $2^ky_0$ respectively, it is then clear that $Q$ is homothetic to $R$ (and so that one has $Q\in\calB$) while one has $Q^*=2^kR^*$. Hence $2^kR^*\in\calB^*$, what we wanted to prove.

Now take a strictly increasing sequence $(\sigma_k)\subseteq \{\sigma(R^*):R\in\calB\}$ verifying $\sigma_k\to+\infty$ and for which $(\sigma_k 2^{-k})$ is a strictly increasing sequence. 

Let $Q_k:=2^k [0,\sigma_0]\times[0,\sigma_0/\sigma_k]$ for all $k\in\N$. Observing first that $\sigma([0,\sigma_0]\times[0,\sigma_0/\sigma_k])=\sigma_k$, translation-invariance of $\calB^*$ ensures that one has $[0,\sigma_0]\times[0,\sigma_0/\sigma_k]\in\calB^*$. By the preceding comments, it hence follows that one has $Q_k\in\calB^*$.

Since now, for all $n\in\N$, the family $\calQ_n:=\{Q_k:0\leq k\leq n\}$ is a finite subset of $\calB^*$ with $n+1$ pairwise incomparable elements up to translation (writing $Q_k:= [0,2^k\sigma_0]\times[0,2^k\sigma_0/\sigma_k]$, it is clear indeed that the $x$-side length of $Q_k$ increases with $k$, while its $y$-side length decreases), it follows that $\calB$ enjoys property (S) of \cite{STOKOLOS1988}, and hence differentiates exactly $L\log L(\R^2)$.
\end{proof}

\begin{proof}[Proof of Corollary~\ref{cor}]
Given $R$ a rectangle whose longest side makes an angle $\theta\in\btheta$ with the horizontal axis, and satisfying $\sigma(R)=\bsigma(\theta)$, we compute using the previous results:
$$
\frac{\check{L}}{\check{\ell}}=\frac{(1-2t)\cos\theta + \frac{\sin\theta}{\bsigma(\theta)}}{(1-2t)\sin\theta +\frac{\cos\theta}{\bsigma(\theta)}},
$$
and:
$$
\frac{\hat{L}}{\hat{\ell}}=\frac{\cos\theta+\frac{\sin\theta}{\bsigma(\theta)}}{\sin\theta+\frac{\cos\theta}{\bsigma(\theta)}}.
$$
Since both of those ratios tend to $+\infty$ when $\theta$ approaches $0$, it follows from Lemma~\ref{lem.2} that $\check{\calB}^{\bsigma}$ and $\hat{\calB}^{\bsigma}$ differentiate exactly $L\log L(\R^2)$. Now use Lemma~\ref{lem.1} to infer, by the Sawyer-Stein principle (see \eqref{eq.maxphi} in the introduction), that $\calB^{\bsigma}$ differentiates exactly $L\log L(\R^2)$.
\end{proof}
\begin{Remark}
A simple computations shows that the shape-function $\bsigma:\btheta\to (1,+\infty)$ constructed before in such a way that $\calB^{\bsigma}$ differentiates $L\log L(\R^2)$, has a linear growth with respect to $1/\theta$ (meaning that there are constants $0<c_1<c_2$ for which one has $c_1/\theta \leq\bsigma(\theta )\leq c_2/\theta$ for all $\theta\in\btheta$ small enough).
\end{Remark}
\section{Shape-functions constant on blocks}

\begin{figure}[h!]
\begin{center}
\begin{tikzpicture}[line cap=round,line join=round,>=triangle 45,x=5.0cm,y=5.0cm]
\draw[->,color=black] (-1.1,0.) -- (1.1,0.);
\foreach \x in {-1.,-0.8,-0.6,-0.4,-0.2,0.2,0.4,0.6,0.8,1.}
\draw[shift={(\x,0)},color=black] (0pt,-2pt);
\draw[->,color=black] (0.,-1.1) -- (0.,1.1);
\foreach \y in {-1.,-0.8,-0.6,-0.4,-0.2,0.2,0.4,0.6,0.8,1.}
\draw[shift={(0,\y)},color=black] (-2pt,0pt);
\clip(-1.1,-1.1) rectangle (1.1,1.1);
\fill[line width=2.pt,color=ttttff,fill=ttttff,fill opacity=0.10000000149011612] (1.,0.11862396204033215) -- (1.,-0.11862396204033215) -- (-1.,-0.11862396204033215) -- (-1.,0.11862396204033215) -- cycle;
\fill[line width=2.pt,color=ccqqqq,fill=ccqqqq,fill opacity=0.10000000149011612] (0.10326251653234646,1.001702798762963) -- (0.3347508576537365,0.9497439168903552) -- (-0.10326251653234646,-1.001702798762963) -- (-0.3347508576537365,-0.9497439168903552) -- cycle;
\fill[line width=2.pt,color=qqzzqq,fill=qqzzqq,fill opacity=0.10000000149011612] (0.656575079380946,0.7635317999311166) -- (0.8169859866059159,0.5887321479753798) -- (-0.656575079380946,-0.7635317999311166) -- (-0.8169859866059159,-0.5887321479753798) -- cycle;
\fill[line width=2.pt,fill=black,fill opacity=1.0] (0.,0.) -- (0.12157540463575348,0.) -- (0.11862396204033215,0.11862396204033215) -- (0.,0.11862396204033215) -- cycle;
\draw [line width=2.pt,color=ttttff] (1.,0.11862396204033215)-- (1.,-0.11862396204033215);
\draw [line width=2.pt,color=ttttff] (1.,-0.11862396204033215)-- (-1.,-0.11862396204033215);
\draw [line width=2.pt,color=ttttff] (-1.,-0.11862396204033215)-- (-1.,0.11862396204033215);
\draw [line width=2.pt,color=ttttff] (-1.,0.11862396204033215)-- (1.,0.11862396204033215);
\draw [line width=2.pt,color=ccqqqq] (0.10326251653234646,1.001702798762963)-- (0.3347508576537365,0.9497439168903552);
\draw [line width=2.pt,color=ccqqqq] (0.3347508576537365,0.9497439168903552)-- (-0.10326251653234646,-1.001702798762963);
\draw [line width=2.pt,color=ccqqqq] (-0.10326251653234646,-1.001702798762963)-- (-0.3347508576537365,-0.9497439168903552);
\draw [line width=2.pt,color=ccqqqq] (-0.3347508576537365,-0.9497439168903552)-- (0.10326251653234646,1.001702798762963);
\draw [line width=2.pt,color=qqzzqq] (0.656575079380946,0.7635317999311166)-- (0.8169859866059159,0.5887321479753798);
\draw [line width=2.pt,color=qqzzqq] (0.8169859866059159,0.5887321479753798)-- (-0.656575079380946,-0.7635317999311166);
\draw [line width=2.pt,color=qqzzqq] (-0.656575079380946,-0.7635317999311166)-- (-0.8169859866059159,-0.5887321479753798);
\draw [line width=2.pt,color=qqzzqq] (-0.8169859866059159,-0.5887321479753798)-- (0.656575079380946,0.7635317999311166);
\draw [line width=2.pt] (0.,0.)-- (0.12157540463575348,0.);
\draw [line width=2.pt] (0.12157540463575348,0.)-- (0.11862396204033215,0.11862396204033215);
\draw [line width=2.pt] (0.11862396204033215,0.11862396204033215)-- (0.,0.11862396204033215);
\draw [line width=2.pt] (0.,0.11862396204033215)-- (0.,0.);
\begin{scriptsize}
\draw[color=ttttff] (0.36275442120732,0.045973831649209596) node {$R_k^1$};
\draw[color=ccqqqq] (0.07976432054964831,0.3401345941749468) node {$R_k^{N_k}$};
\draw[color=qqzzqq] (0.28828334208688006,0.23959863736235304) node {$R_k^i$};
\draw[color=black] (-0.039600974394879572,0.031079615825121636) node {$\Theta$};
\end{scriptsize}
\end{tikzpicture}\caption{The rectangles $R_k^i$, $1\leq i\leq N_k$ (after rotation of angle $-\theta_{k+1}$ around the origin) in the proof of Proposition~\ref{prop.orl}}\label{fig.2}
\end{center}
\end{figure}
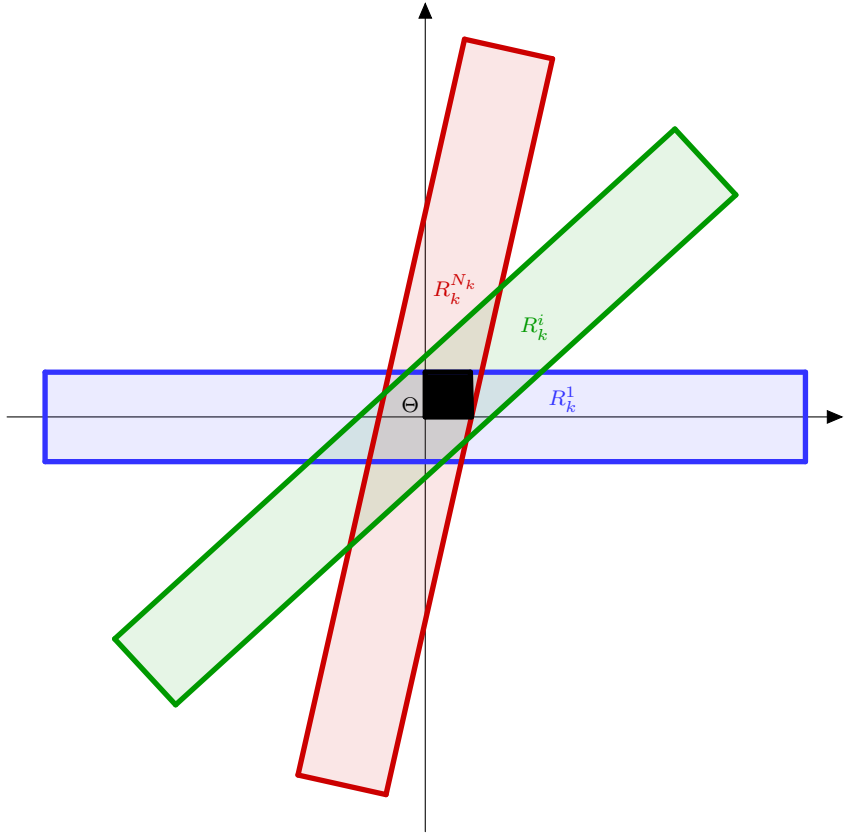

We now examine the case where $$\btheta=\left\{\theta_{k+1}+ \frac{i-1}{N_k} (\theta_k-\theta_{k+1}):k\in\N, 1\leq i\leq N_k\right\}
$$ is associated to a sequence $(\theta_k)_{k\in\N}\subseteq (0,\pi/4)$ decreasing to $0$ and to a sequence of integers $(N_k)_{k\in\N}\subseteq\N^*$ by inserting uniformly $N_k$ angles in between $\theta_{k+1}$ and $\theta_k$, and where $\bsigma:\btheta\to [1,+\infty)$ is constant on each of those ``blocks'', meaning that for each $k\in\N$, there exists a real number $\sigma_k\geq 1$ such that one has $\bsigma(\theta_{k+1}+\frac{i-1}{N_k} (\theta_k-\theta_{k+1}))=\sigma_k$ for all $1\leq i\leq N_k$. We remain in this setting until the end of this section, unless otherwise mentioned.

Recall that, given sequences $(a_k)\subseteq\R_+$ and $(b_k)\subseteq\R_+^*$, one writes $a_k=o(b_k)$ (resp. $a_k=O(b_k)$) if the quotient $\frac{a_k}{b_k}$ tends to $0$ (resp. remains bounded) as $k$ grows to $\infty$.
\begin{Proposition}\label{prop.orl}
Assume $\btheta$ and $\bsigma$ are as before and that one has moreover, for each $k\in\N$:
\begin{equation}\label{eq.angle}
\sin\left(\frac{\theta_k-\theta_{k+1}}{N_k}\right)\geq\frac{4}{\sigma_k}.
\end{equation}
Let $\Phi:[0,+\infty)\to [0,+\infty)$ be a Young function.
If the homothecy-invariant basis $\calB^{\bsigma}$ associated to $\sigma$ differentiates the Orlicz space $L^\Phi(\R^2)$, then one has $N_k\sigma_k=O[\Phi(\sigma_k)]$ as $k\to\infty$.
\end{Proposition}
\begin{proof}
We follow a similar strategy to the one developed by the second author in \cite{MOONENS2016}. To this purpose, define for each $k\in\N$ a two-dimensional interval $Q_k:=[-\sigma_k,\sigma_k]\times [-1,1]$.
Define, for $1\leq i\leq N_k$, $R_k^i$ to be the rectangle obtained from $Q_k$ by rotating it around the origin by an angle $\theta_{k+1}+\frac{i-1}{N_k}(\theta_k-\theta_{k+1})$.
It is not hard to observe that condition \eqref{eq.angle} ensures that one has:
$$
\left|\bigcup_{i=1}^{N_k} R_k^i\right|\geq\frac 12 N_k|R_k|,
$$
by noting for example that \eqref{eq.angle} implies that (at least) half (in area) of the rectangle $R_k^i$ has no overlap with $R_k^j$ for $j\neq i$. Define then $Y_k:=\bigcup_{i=1}^{N_k} R_k^i$ and let $\Theta_k$ be the set obtained by rotating $[0,1]^2$ around the origin by an angle $\theta_{k+1}$ (see Figure~\ref{fig.2}).

Given $k\in\N$ and $x\in Y_k$, there is an $1\leq i\leq N_k$ such that one has $x\in R_k^i$. But since one has $R_k^i\in\calB^{\bsigma}$, this yields:
$$
M_{\calB^{\bsigma}} (\sigma_k \chi_{\Theta_k})\geq \frac{1}{|R_k^i|}\int_{R_k^i} \sigma_k \chi_{\Theta_k}=\frac{\sigma_k}{|R_k^i|}=\frac 14 .
$$
We hence have $Y_k\subseteq\{x\in\R^2: M_{\calB^{\bsigma}} (\sigma_k \chi_{\Theta_k})\geq \frac 14 \}$. Yet if $\calB^{\bsigma}$ did differentiate $L^\Phi(\R^2)$, it would follow from \eqref{eq.maxphi} that one would have:
\begin{multline*}
\frac 12 N_k 4\sigma_k=\frac 12 N_k |R_k|\leq |Y_k|\\\
\leq |\{x\in\R^2: M_{\calB^{\bsigma}} (\sigma_k \chi_{\Theta_k})\geq \frac 14 \}|\leq 4C\int_{\R^2} \Phi(\sigma_k\chi_{\Theta_k})=4C\Phi(\sigma_k),
\end{multline*}
from which the announced statement follows.
\end{proof}

The previous result has amusing consequences when one considers sets of the previous form associated to the sequence defined by $\theta_k:=2^{-k}$ (recall from the introduction that for this sequence $\btheta_0:=\{2^{-k}:k\in\N\}$ it is known that $\calB_{\btheta_0}$ differentiates $L^p(\R^2)$ for all $p>1$, with no restriction on the shapes of the rectangles, while it is unknown if it differentiates $L\log L(\R^2)$). The corollary below shows that adding angles in between two terms of the ``model'' geometric sequence, while restricting the shape in those ``blocks'', may fail to differentiate Orlicz spaces lying in between $L^1(\R^2)$ and $L^p(\R^2)$ for all $1<p<\infty$~---~of course, the shape-function therefore has to increase quite fast.
\begin{Corollary}
Assume that $\theta_k:=2^{-k}$ for all $k\in\N$ and let $\btheta$ be associated to $(\theta_k)_{k\in\N}$ and to $(N_k)_{k\in\N}\subseteq\N^*$ as before. Define also, for $k\in\N$, $\sigma_k:=4/\sin(2^{-k-1}/N_k)$ and let $\bsigma$ be the associated shape-function.
If, for some $\alpha>0$, one has $k^\alpha=o(N_k)$ when $k\to\infty$, then $\calB^{\bsigma}$ fails to differentiate $L\log^\alpha L(\R^2)$.
\end{Corollary}
\begin{proof}
This corollary follows in a straightforward way by contradiction from the previous proposition applied to $\Phi(t):=t(1+\log_+^\alpha t)$.
\end{proof}
\begin{Remark}
One can of course formulate a similar corollary, starting from a lacunary sequence $(\theta_k)_{k\in\N}$ as in \cite{MOONENS2016} instead of the ``model'' sequence $(2^{-k})_{k\in\N}$.
\end{Remark}

\end{document}